\PassOptionsToPackage{table}{xcolor}
\documentclass{amsart}
\usepackage{mathptmx, amscd, amssymb, amsmath, enumerate, slashbox, hyperref, float,mathtools, tikz-cd,thmtools, thm-restate}
\usepackage[table]{xcolor}

\DeclareMathAlphabet{\mathcal}{OMS}{cmsy}{m}{n}
\newtheorem{theorem}{Theorem}[section]

\newtheorem{lem}[theorem]{Lemma}
\newtheorem{prop}[theorem]{Proposition}

\theoremstyle{definition}

\numberwithin{equation}{theorem}

\def\phi{\varphi}

\def\to{\longrightarrow}

\def\dim{\operatorname{dim}}
\def\codim{\operatorname{codim}}
\def\height{\operatorname{height}\,}
\def\rank{\operatorname{rank}}
\def\Hom{\operatorname{Hom}}

\def\Ext{\operatorname{Ext}}

\def\Sing{\operatorname{Sing}}
\def\Spec{\operatorname{Spec}\,}

\def\min{\operatorname{min}}

\def\PP{\mathbb{P}}

\newcommand{\mf}[1]{\mathfrak{#1}}

\def\calF{\mathcal{F}}

\def\calO{\mathcal{O}}

\definecolor{myBlue}{HTML}{0072B2}
\definecolor{myOrange}{HTML}{E69F00}
\definecolor{myVermillion}{HTML}{D55E00}
\definecolor{myGreen}{HTML}{009E73}
\definecolor{babyBlossoms}{HTML}{cc79a7}

\begin{document}
\title{Isomorphisms Between Local Cohomology Modules as Truncations of Taylor Series}

\author{Jennifer Kenkel}
\address{Department of Mathematics, University of Kentucky, 719 Patterson Office Tower, Lexington, KY~40507, USA}
\email{jennifer.kenkel@gmail.com}

\thanks{The author was supported by NSF grant DMS~1246989. }

\begin{abstract}
Let $R$ be a standard graded polynomial ring that is finitely generated over a field, and let $I$ be a homogenous prime ideal of $R$. Bhatt, Blickle, Lyubeznik, Singh, and Zhang examined the local cohomology of $R/I^t$, as $t$ grows arbitrarily large. Such rings are known as thickenings of $R/I$. We consider $R = \mathbb{F}[X]$ where $\mathbb{F}$ is a field of characteristic 0, $X$ is a $2 \times m$ matrix, and $I$ is the ideal generated by size two minors. We give concrete constructions for the local cohomology modules of thickenings of $R/I$. Bizarrely, these local cohomology modules can be described using the Taylor series of natural log.  
\end{abstract}
\maketitle

\section{Introduction}
\label{section:introduction}

\par 

Let $R$ be a graded Noetherian commutative ring and $I$ a homogeneous ideal of $R$. For each integer $t > 1$, the rings $R/I^t$ are referred to as \emph{thickenings} of $R/I$. 
The canonical surjection from $R/I^{t+1}$ to $R/I^{t}$ induces a degree-preserving map on local cohomology modules: 
$$H^k_{\mathfrak{m}} (R/I^{t+1})_j \to  H^k_{\mathfrak{m}} (R/I^{t})_j$$ 
 Our focus is on local cohomology modules supported in the maximal ideal, $\mf{m}$.  
\par In general, much of the work that has been done on local cohomology modules examines whether the module is or is not zero.
 While they are useful, local cohomology modules are defined homologically, and thus are difficult to work with concretely. They tend to be large and, even when derived from simple rings, are rarely explicitly described.
 In a paper of Bhatt, Blickle, Lyubeznik, Singh, and Zhang, the authors examined when the induced maps on local cohomology modules in a fixed degree are isomorphisms for large values of $t$ \cite{BBLSZ}. Their work is part of a growing movement towards studying how local cohomology modules behave \cite{BBLSZ, DaoAndMontano, DaoAndMontano2, DeStefani}.
 \par In this paper, we will build on the work of \cite{BBLSZ}. We consider the induced maps on local cohomology modules in the case of a determinantal ring.  Consider a power series with variables in the ideal $I$. In the ring $R/I^t$, such an infinite series behaves like a finite sum, as all but finitely many terms are zero. We use this idea to understand isomorphisms between local cohomology modules as truncations of a power series. In particular, we show that in characteristic 0, $H^3_{\mf{m}}(R/I^t)_0$ is generated as a vector space over the base field by an element represented by the Taylor series of natural log.  This gives an explicit description of the isomorphisms guaranteed in \cite{BBLSZ}. We then construct elements of the same local cohomology module over a field of characteristic $p>0$, and show that the rank of $H^3_{\mf{m}}(R/I^t)_0$ grows arbitrarily large along an infinite subsequence of natural numbers $t$. 
 \par This paper is structured as follows: 
\begin{enumerate}
	\item In Section \ref{BBLSZapplies!}, we set the scene by introducing a determinantal variety over a field of characteristic 0.  We recall the results of \cite{BBLSZ} and verify this variety satisfies the hypotheses there given. 
	\item In Section \ref{Section:RankIsOne}, we introduce our protagonist, the local cohomology module of thickenings of a determinantal variety. We show that, in degree 0, the vector space rank of this module is 1.  
	\item In Section \ref{detlIsoSection}, we explicitly construct an element of the local cohomology described in Section \ref{Section:RankIsOne}. Based on our results from Section \ref{Section:RankIsOne}, we are ensured that we have a basis for this module as a vector space. 
	\item In Section \ref{charPSection}, we consider the same variety over a field of positive characteristic.
	We explicitly construct elements of the module $H^3_{\mf{m}}(R/I^t))_0$ for all $t>1$. 
	In that context, isomorphisms do not exist.
\end{enumerate}

\section{Isomorphisms Are Guaranteed to Exist} \label{BBLSZapplies!}
\par 	
For the entirety of this paper, we examine the ring $R= \mathbb{F}[X]$ where $X$ is a $2 \times 3$ matrix of indeterminates and $\mathbb{F}$ is a field: 

$$ R = \mathbb{F}\begin{bmatrix} u & v & w \\ x & y & z \end{bmatrix}. $$ 
Let $I$ be the ideal generated by the size two minors of the matrix $X$, that is, the elements; 
$$  \Delta_1 = vz-wy, \quad \Delta_2 = wx-uz, \quad \Delta_3 = uy-vx  .$$ 
The field $\mathbb{F}$ will be of characteristic 0, except in Section \ref{charPSection}. 
\par In this section, we will repeat the theorem of \cite{BBLSZ} that guarantees an isomorphism between local cohomology modules of thickenings. We will then verify that the ring $R$ and the ideal $I$ satisfy the hypotheses when the cohomological index, $k$, is equal to 3.

\begin{theorem} \cite{BBLSZ}
	\label{BBLSZ:AG}
	Let $X$ be a closed lci subvariety of $\PP^n$ over a field of characteristic 0, defined by a sheaf of ideals $\mathcal{I}$. Let $X_t \subset \mathbb{P}^n$ be the $t$-th thickening of $X$, i.e., the closed subscheme defined by the sheaf of ideals $\mathcal{I}^t$. Let~$\calF$ be a coherent sheaf on $\PP^n$ that is flat along $X$. Then, for each $k<\codim(\Sing X)$, the natural map
	\[
	H^k(X_{t+1},\calO_{X_{t+1}}\otimes_{\calO_{\PP^n}}\calF) \to H^k(X_t,\calO_{X_t}\otimes_{\calO_{\PP^n}}\calF)
	\]
	is an isomorphism for all $t\gg0$.
	In particular, if $X$ is smooth or has at most isolated singular points, then the map above is an isomorphism for $k<\dim X$ and $t\gg0$.
\end{theorem}

Rather than considering this theorem in the setting of sheaf cohomology, as it originally appears, we consider the local cohomology setting:
 
\begin{theorem}\cite{BBLSZ} 
	\label{BBLSZ:CA}
	Let $R$ be a standard graded polynomial ring that is finitely generated over a field~$R_0$ of characteristic 0.  Let $\mf{m}$ be the homogeneous maximal ideal of $R$. Let $I$ be a homogeneous prime ideal such that $R/I$ is a locally complete intersection on the set $\Spec{R} \backslash \{\mf{m}\}$, and let $k$ be a natural number such that $k<\dim(R)-~\height(\Sing R/I)$. Then, for a fixed natural number $j$, the maps between the modules $H^k_{\mf{m}}(R/I^{t+1})_j$ and $H^k_{\mf{m}}(R/I^t)_j$ are isomorphisms for sufficiently large $t$.   
\end{theorem} 

We first determine in which indices the hypotheses of Theorem \ref{BBLSZ:CA} apply. 
For a general treatment of determinantal rings, see \cite{DetRings}. In particular, Proposition 1.1 of \cite{DetRings} gives that the ring $R/I$ has dimension 4, and Theorem 2.6 gives that the localization of $R/I$ at a prime ideal $ \mf{p} $ is regular if and only if $\mf{p}$ is not equal to the maximal ideal $\mf{m}$. 

Since $R/I$ is regular when localized away from the maximal ideal, the ring $R/I$ is a locally complete intersection on the punctured spectrum. On the other hand, $(R/I)_{\mf{m}}$ is not regular, so we have that $\Sing(R/I) = { \mf{m} }$ and thus $\height \left( \Sing(R/I) \right)$ is the dimension of $R/I$, which is 4. Therefore Theorem \ref{BBLSZ:CA} applies for cohomological indices $k \leq 3$.

\section{The rank of $H^{3}_{\mf{m}}(R/I^t)_0$ is one}\label{Section:RankIsOne}

\par Hochster and Eagon showed that determinantal rings are Cohen-Macaulay
in \cite{HochsterAndEagon}, specifically, in our case, the local cohomology modules $H^k_{\mf{m}}(R/I)$ are zero at every cohomological index  $k \neq 4$.  
However, the successive thickenings, $R/I^t$, are not Cohen-Macaulay for all $t$ greater than 1; see \cite{DEP}. 

As is stated in Proposition 7.24 of \cite{DetRings}, the depth of the ring $R/I^t$ is at least three for all $t$. Since the rings $R/I^t$ are not Cohen-Macaulay for all $t \geq 2$, this implies that the depth of $R/I^t$ must be exactly three. Therefore the module $H^3_{\mf{m}}(R/I^t)$ is nonzero for all thickenings with $t \geq 2$.  As we shall see, $H^3_{\mf{m}}(R/I^t)_0$ is a rank 1 vector space over $\mathbb{F}$ for each $t \geq 2$. Towards proving this, we first recall the following theorem.
\begin{theorem} \label{LSW} \cite{LSW} 
	Let $R=\mathbb{Z}[X]$ be a polynomial ring, where $X$ is an $m \times n$ matrix of indeterminates. Let $I_d$ be the ideal generated by the size $d$ minors of $X$. If $2 \leq d \leq \min ( m, n )$ and $d$ differs from at least one of $m$ and $n$, then there exists a degree-preserving isomorphism 
	$$H^{mn-d^2+1}_{I_d}(\mathbb{Z}[X]) \cong H^{mn}_{\mf{m}}(\mathbb{Q}[X]).$$ 
\end{theorem}
	 Applying Theorem \ref{LSW} when $X$ is an $2 \times 3$ matrix of indeterminates and $I$ is the ideal generated by size $2$ minors of $X$ gives: 
	\begin{equation} \label{LSW:Ex} 
	H^3_{I}(\mathbb{Z}[X]) \cong H^6_{\mf{m}}(\mathbb{Q}[X])
	\end{equation} 
\par We are considering the ring $\mathbb{F}[X]$ where $\mathbb{F}$ is some field of characteristic 0, not the ring $\mathbb{Z}[X]$. However, we claim the above isomorphism implies that the modules $H^3_I(\mathbb{F}[X])$ and $H^6_{\mf{m}}(\mathbb{F}[X])$ are isomorphic. To see this, first tensor both sides of Equation \ref{LSW:Ex} with the module $\mathbb{F}[X]$ to get 
\begin{equation} \label{littlePaperTensor}
\mathbb{F}[X] \otimes_{\mathbb{Z}[X]}  H^{3}_{I}(\mathbb{Z}[X]) \cong \mathbb{F}[X] \otimes_{\mathbb{Z}[X]} H^{6}_{\mathfrak{m}}(\mathbb{Q}[X]) . 
\end{equation} 
Note that, since $\mathbb{F}$ is a field of characteristic 0, the module $\mathbb{F}[X]$ is flat over $\mathbb{Z}[X]$.  

\begin{lem}  \cite{24hours} \label{tensoring} Let $I$  be an ideal of a ring $R$, and let $M$ be an $R$-module.  Then if $R \to S$ is flat, there is a natural isomorphism of $S$-modules 
	\begin{equation*}   S \otimes_R H^j_{I} (M) \cong H^j_{IS}(S \otimes_R M) .
	\end{equation*} \end{lem} 
\noindent Applying Lemma \ref{tensoring} to Equation \ref{littlePaperTensor} gives that 
\begin{align*} \mathbb{F}[X] \otimes_{\mathbb{Z}[X]} H^3_{I}(\mathbb{Z}[X]) &\cong H^3_{I\mathbb{F}[X]}(\mathbb{F}[X] \otimes_{\mathbb{Z}[X]} \mathbb{Z}[X])\\ 
&\cong H^3_I (\mathbb{F}[X]). 
\end{align*} 
Similarly 
\begin{align*} 
\mathbb{F}[X] \otimes_{\mathbb{Z}[X]} H^{6}_{\mathfrak{m}}(\mathbb{Q}[X]) &\cong H^6_{\mf{m}}(\mathbb{F}[X] \otimes_{\mathbb{Z}[X]} \mathbb{Q}[X] ) 
\\ &\cong H^6_{\mf{m}}(\mathbb{F}[X]).
\end{align*}  
\begin{prop} \label{rankOne} The module $H^3_{\mf{m}}(R/I^t)_0$ is an $\mathbb{F}$-vector space of rank 1 for all $t$. \end{prop} 
\begin{proof} 
	\par 
	Let $\omega_R$ denote the canonical module of $R$ and let $E_R(R/\mf{m})$ denote the injective hull of the residue field, $R/\mf{m}$. As $R$ is a Gorenstein ring of dimension 6, the canonical module $\omega_R$ is isomorphic to $R(-6)$. The injective hull $E_R(R/\mf{m})$ is isomorphic to $H^6_{\mf{m}}(R)(-6)$. 

	\index{Local duality} Local duality gives that 
	$$H^3_{\mf{m}}(R/I^t)^{\vee} \cong \Ext^{3}_R(R/I^t, \omega_R), $$
	where $(-)^{\vee}$ indicates $\Hom_{R}(-, E_R(R/\mf{m}))$. 
	\par An $R$-homomorphism in $\Hom_R(H^3_{\mf{m}}(R/I^t), E_R(R/\mf{m}))_j$ is determined by the preimage of elements in $E_R(R/\mf{m})$ of degree zero. Therefore the rank of $\Hom_R(H^3_{\mf{m}}(R/I^t), E_R(R/\mf{m}))_j$ is equal to the rank of $H^3_{\mf{m}}(R/I^t)_{-j}$.  
	We thus have 
	\begin{align*} 
	\rank  H^3_{\mf{m}}(R/I^t)_j &= \rank \left( H^3_{\mf{m}}(R/I^t)^{\vee}\right)_{-j} \\
	&= \rank( \left(  \Ext^{3}_R(R/I^t, \omega_R) \right)_{-j} ) \\
	&= \rank ( \left ( \Ext^{3}_R(R/I^t, R(-6)) \right)_{-j} \\
	&= \rank ( \left ( \Ext^{3}_R(R/I^t, R) \right)_{-j-6} ,
	\end{align*} 
	i.e., the rank of $H^3_{\mf{m}}(R/I^t)_j$ is the same as the rank of $\Ext^{3}_R(R/I^t, R)_{-(j+6)}$. 
	Note that one can define a different local cohomology module as the direct limit of these $\Ext$-modules : 
	$$ H^3_I(R) = \lim\limits_{t \to \infty }\Ext^{3}_R(R/I^t, R).$$ 
	Since Theorem~\ref{BBLSZ:CA} guarantees an eventual isomorphism, the rank of the module  $H^3_{\mf{m}}(R/I^t)_j$  must equal the rank of $H^3_I(R)_{j+6}$ for sufficiently large $t$. 
	Therefore, we can compute the rank of $H^3_{\mf{m}}(R/I^t)_j$ as an $\mathbb{F}$-vector space by instead calculating the rank of $H^3_I(R)$ with a degree shift: 
	\begin{equation} \label{rankOfThickening} 
	\rank H^3_{\mf{m}}(R/I^t)_j = \rank H^3_I(R)_{-j-6} \end{equation}
	for sufficiently large $t$.   

	Recall that Equation~\ref{littlePaperTensor} gave $H^3_I(\mathbb{F}[X]) \cong H^6_{\mf{m}}(\mathbb{F}[X])$ as graded modules. 
	Therefore, the rank of $H^3_I(\mathbb{F}[X])$ in any degree is equal to the rank of  $H^6_{\mf{m}}(\mathbb{F}[X])$ in that degree. Since the rank of $H^3_{\mf{m}}(R/I^t)_j$ equals the rank of $H^3_I(R)_{-j-6}$ from Equation~\ref{rankOfThickening}, we have 
	$$ \rank H^3_{\mf{m}}(R/I^t)_j = \rank H^6_{\mf{m}}(R)_{-j-6}$$
	In particular, when $j=0$ 
	$$ \rank H^3_{\mf{m}}(R/I^t)_0 = \rank H^6_{\mf{m}}(R)_{-6}.$$
	As $R$ is a regular ring in six variables, the top local cohomology module is well-understood: 
	$$H^6_{\mf{m}}(R) = \frac{1}{uvwxyz} \mathbb{F}[ u^{-1}, v^{-1}, w^{-1}, x^{-1}, y^{-1}, z^{-1}].$$ 
	It follows from the above description that the rank of $H^6_{\mf{m}}(R)_{-6}$ equals 1 as an $\mathbb{F}$-vector space. 
\end{proof}
\section{An Element of Local Cohomology Is Described Using Natural Log} \label{detlIsoSection} 

One can determine the module $H^3_{\mf{m}}(R/I^t)_0$ by considering elements in the \v{C}ech complex on the generators 
$${\bf x}=u, v, w, x, y, z $$ 
of the maximal ideal, $\mf{m}$. As $R/I$ is a four dimensional ring, the maximal ideal $\mf{m}$ could be generated up to radical by only four elements. However, we chose to use six variables, for the sake of symmetry.
\par Thus, we expect to find an element of $\check{C}^{3}({\bf x},R/I^t)$, that is, an element of
$$(R/I^t)_{uvw} \oplus  (R/I^t)_{uvx} \oplus (R/I^t)_{uvy} \oplus \hdots \oplus (R/I^t)_{xyz} $$ that maps to $0$ in $\check{C}^4({\bf x}, R/I^t)$ but is not in the image of $\check{C}^2(\textbf{x},R/I^t)$. Note that the presence of only two variables in the denominator of a particular component of $\check{C}^3(\textbf{x}, R/I^t)$ does not guarantee that the element is in the image of $\check{C}^2(\textbf{x}, R/I^t)$, as we shall see.

Surprisingly, the isomorphism $H^3_{\mf{m}}(R/I^{t+1})_0 \to H^3_{\mf{m}}(R/I^t)_0$ can be elegantly understood in terms of truncations of the formal power series of natural log.  
\begin{theorem}	\label{log} Let $R=\mathbb{F}\begin{bmatrix}
	u & v & w \\ x & y & z
	\end{bmatrix}$, i.e.,  $R$ is the ring of polynomials in 6 indeterminates over a field of characteristic 0. Let $I$ be the ideal generated by size two minors of $\begin{bmatrix} u & v & w \\ x& y & z \end{bmatrix}$ and let 
	$$  \Delta_1 = vz-wy, \quad \Delta_2 = wx-uz, \quad \Delta_3 = uy-vx .$$
	Then the identity $$ \ln\left(\frac{wy}{vz} \frac{uz}{wx} \frac{vx}{uy}\right) = \ln(1) = 0$$ gives the following identity in the fraction field of the $I$-adic completion of $R$:  
	$$ \displaystyle\sum\limits_{m=1}^{\infty} \frac{1}{m} \left(\frac{\Delta_1}{vz}\right)^m  
	+ \sum\limits_{m=1}^{\infty} \frac{1}{m} \left(\frac{\Delta_2}{wx}\right)^m
	+ \sum\limits_{m=1}^{\infty} \frac{1}{m} \left(\frac{\Delta_1}{uy}\right)^m = 0 .$$ 
	The $t^{th}$ truncation of this Taylor series yields the generator for $H^3_{\mf{m}}(R/I^t)_0.$ 
\end{theorem}
\par While the above identity in the fraction field of the $I$-adic completion is the heart of our local cohomology module element, it must be finessed slightly into the form of an element in $\check{C}^3(\bf{x}, R/I^t)$. 
We will show that Table \ref{charzeroelement} describes such an element, which we will henceforth refer to as $\eta$.  We will prove the above theorem in two steps. First, in Subsection \ref{subsection:char0vanishing}, we will show that $\eta$ maps to 0 in $\check{C}^4({\bf x}, R/I^t)$. Second, in Subsection \ref{subsection:char0notImage}, we will show that $\eta$ is not in the image of $\check{C}^2({\bf x}, R/I^t)$. 

{	\rowcolors{2}{babyBlossoms!50}{white}
\begin{table}[t]
	\begin{center}
		\caption{Element in $H^3_{\mf{m}}(R/I^t)$}
		\label{charzeroelement}
		\medskip
		\begin{tabular}{lc|r}
			\hline
			\rowcolor{babyBlossoms} \hspace{.7in} 
			 & \multicolumn{1}{c}{Component of $\check{C}^{3}({\bf x},R/I^t)$} & \multicolumn{1}{c}{Component in the $I$-adic completion}  \\
			\hline
			$R_{uvw}$ & $0$ & $0$\\ 
				$R_{uvx}$ & $0$& $0$\\
			$R_{uwx}$ & $0$ & $0$\\ 
				$R_{vwx}$ & $0$& $0$\\ 
			$R_{uvy}$ & $0$ & $0$\\ 
				$R_{uwy}$ & $0$& $0$ \\ 
			$R_{vwy}$ & $0$& $0$ \\ 
				$R_{uxy}$ & $-\sum\limits_{m=1}^{t-1} \frac{1}{m} (\frac{\Delta_3}{uy})^{m}$ & $\ln\left(\frac{uy}{vx}\right)$  \\ 
			$R_{vxy}$ & $\sum\limits_{m=1}^{t-1} \frac{1}{m} (\frac{\Delta_3}{vx})^{m}$& $\ln\left(\frac{uy}{vx}\right)$  \\ 
							$R_{wxy}$ & $\sum\limits_{m=1}^{t-1} \frac{1}{m} (\frac{\Delta_2}{wx})^{m} - \sum\limits_{k=1^{t-1}} \frac{1}{k} (\frac{\Delta_1}{wy})^{m}$& $\ln\left(\frac{uz}{wx}\right) + \ln\left(\frac{wy}{vz}\right)$\\ 
			$R_{uvz}$ & $0$& $0$\\ 
				$R_{uwz}$ & $0$ & $0$\\ 
			$R_{vwz}$ & $0$ &$0$\\ 
				$R_{uxz}$ & $-\sum\limits_{m=1}^{t-1} \frac{1}{m} (\frac{\Delta_2}{uz})^{m}$& $\ln(\frac{uz}{wx})$ \\ 
			$R_{vxz}$ & $-\sum\limits_{m=1}^{t-1} \frac{1}{m} (\frac{\Delta_1}{vz})^{m} + \sum\limits_{k=1^{t-1}} \frac{1}{k} (\frac{\Delta_3}{vx})^{m}$  & $\ln\left(\frac{vz}{wy}\right) + \ln\left(\frac{uy}{vx}\right)$ \\ 
				$R_{wxz}$ & $\sum\limits_{m=1}^{t-1} \frac{1}{m} (\frac{\Delta_2}{wx})^{m}$& $\ln\left( \frac{uz}{wx}\right) $  \\
			$R_{uyz}$ & $\sum\limits_{m=1}^{t-1} \frac{1}{m} (\frac{\Delta_3}{uy})^{m}  -\sum\limits_{k=1^{t-1}} \frac{1}{k} (\frac{\Delta_2}{uz})^{m}$ & $\ln\left(\frac{vx}{uy} \right) + \ln\left(\frac{uz}{wx} \right) $ \\ 
				$R_{vyz}$ & $-\sum\limits_{m=1}^{t-1}\frac{1}{m} (\frac{\Delta_1}{vz})^{m}$& $\ln\left( \frac{vz}{wy}\right)$  \\ 
			$R_{wyz}$ & $\sum\limits_{m=1}^{t-1} \frac{1}{m} (\frac{\Delta_1}{wy})^{m}$ & $\ln\left(\frac{vz}{wy}\right) $\\ 
				$R_{xyz}$ & $0$ & $0$\\ 
		\end{tabular}
	\end{center}
\end{table}  
}	

\subsection{The Element $\eta$ Vanishes}\label{subsection:char0vanishing} 

	First we show that the element given by Table \ref{charzeroelement} is a cocycle, that is, the image of $\eta$ is 0 in $\check{C}^4({\bf x}, R/I^t)$.

\begin{proof}  Let $\hat{R}$ denote the $I$-adic completion of $R$. Then, in the fraction field of $\hat{R}$, one has
	\begin{align*} 0 & = \ln(1) \\
	&= \ln \left( \textcolor{myVermillion}{\frac{wy}{vz}} \textcolor{babyBlossoms}{ \frac{uz}{wx}} \textcolor{myBlue}{ \frac{vx}{uy}} \right) \\ 
	&=  \ln \left( \left(1-\textcolor{myVermillion}{\left(1- \frac{wy}{vz}\right)}\right)\left(1-\textcolor{babyBlossoms}{\left(1-
		\frac{uz}{wx}\right)}\right)\left(1-\textcolor{myBlue}{\left(1- \frac{vx}{uy}\right)}\right) \right) 
	\\ 
	&= \ln \left( \left(1-\textcolor{myVermillion}{\left( \frac{vz - wy}{vz}\right)}\right)\left(1-\textcolor{babyBlossoms}{\left(
		\frac{wx-uz}{wx}\right)}\right)\left(1-\textcolor{myBlue}{\left( \frac{uy-vx}{uy}\right)}\right) \right) 
	\\
	&=  \ln \left( \left(1-\textcolor{myVermillion}{\left(\frac{\Delta_1}{vz}\right)}\right)\left(1-\textcolor{babyBlossoms}{\left(
		\frac{\Delta_2}{wx}\right)}\right)\left(1-\textcolor{myBlue}{\left(\frac{\Delta_3}{uy}\right)}\right) \right) 
	\\
	&=  \ln \left( \left(1-\textcolor{myVermillion}{\left(\frac{\Delta_1}{vz}\right)}\right)+\ln\left(1-\textcolor{babyBlossoms}{\left(
		\frac{\Delta_2}{wx}\right)}\right)+ \ln \left(1-\textcolor{myBlue}{\left(\frac{\Delta_3}{uy}\right)}\right) \right) 
	\\
	&= \displaystyle\sum\limits_{m=1}^{\infty} \frac{1}{m} \left(\textcolor{myVermillion}{\frac{\Delta_1}{vz}}\right)^m  
	+ \sum\limits_{m=1}^{\infty} \frac{1}{m} \left(\textcolor{babyBlossoms}{\frac{\Delta_2}{wx}}\right)^m
	+ \sum\limits_{m=1}^{\infty} \frac{1}{m} \left(\textcolor{myBlue}{\frac{\Delta_3}{uy}}\right)^m.
	\end{align*} 
	Because the minors, $\Delta_i$, are in the ideal $I$, in $\check{C}^3(\textbf{x}, R/I^t)$, the elements of the form $$\sum\limits_{m=1}^{t-1} \frac{1}{m} \left( \frac{\Delta_1}{vz }\right)^m$$ are exactly equal to the truncations of the sum, $$\sum\limits_{m=1}^{\infty} \frac{1}{m} \left( \frac{\Delta_1}{vz }\right)^m.$$ Therefore, an element of  $\check{C}^3(\textbf{x}, R/I^t)$ has the same image in $\check{C}^4(\textbf{x}, R/I^t)$ as would an element in $\check{C}^3(\textbf{x}, \lim_{t \to \infty} R/I^t)$, that is, an element in the \v{C}ech complex of the $I$-adic completion of $R$.  We also record the completion of the element in Table \ref{charzeroelement}.
	\par The element $\eta$ of Table~\ref{charzeroelement} is nonzero only in components of $\check{C}^3({\bf x}, R/I^t)$ with one of the variables, $u, v,$ and $w$ (i.e., variables in the first row) and two of the variables $x, y,$ or $z$  (i.e., variables in the second row) inverted. Thus, the image of $\eta$ in all components of $\check{C}^4({\bf x}, R/I^t)$ with only one of the variables $x, y,$ or $z$ inverted will certainly be 0; that is, the image of $\eta$ is 0 in the components $(R/I^t)_{uvwx}, (R/I^t)_{uvwy}$ and $(R/I^t)_{uvwz}$. 
	\par By symmetry, it suffices to check that the image of $\eta$ is 0 in the three components, $(R/I^t)_{uvxy}, (R/I^t)_{uxyz},$ and $(R/I^t)_{uwxy}$.
	\begin{enumerate}[1.]
		\item 

	In $(R/I^t)_{uvxy}$, $\eta$ maps to: 

	\begin{equation*}
	\sum\limits_{m=1}^{t-1}\frac{1}{m}\left(\frac{\Delta_3}{uy}\right) + \sum\limits_{m=1}^{t-1}\frac{1}{m}\left(\frac{\Delta_3}{vx}\right),
	\end{equation*} 
	which agrees with the infinite sum 
	\begin{equation*} \sum\limits_{m=1}^{\infty}\frac{1}{m}\left(\frac{\Delta_3}{uy}\right) + \sum\limits_{m=1}^{\infty}\frac{1}{m}\left(\frac{\Delta_3}{vx}\right)
	\end{equation*} 
	in the $I$-adic completion of the module $(R/I^t)_{uvxy}$. We have that the sum equals
	\begin{align*} 
	\ln\left(\frac{vx}{uy}\right) + \ln\left( \frac{uy}{vx} \right) 
	=\ln\left(\frac{vx}{uy}\right) - \ln \left( \frac{vx}{uy} \right) 
	= 0.
	\end{align*} 
\item 	The image of $\eta$ in $(R/I^t)_{uxyz}$ is 
	\begin{align*}
	&\sum_{m=1}^{t-1} \frac{1}{m}\left( \frac{\Delta_3}{uy}\right)^m - \sum_{m=1}^{t-1}\frac{1}{m}\left( \frac{\Delta_2}{uz}\right)^m -  \sum\limits_{m=1}^{t-1} \frac{1}{m} \left(\frac{\Delta_3}{uy}\right)^{m}  + \sum\limits_{m=1}^{t-1} \frac{1}{m} \left(\frac{\Delta_2}{uz}\right)^{m} = 0
	\end{align*} 
\item	Finally, the image of $\eta$ in $(R/I^t)_{uwxy}$ is 
	\begin{align*}
	\sum\limits_{m=1}^{t-1} \frac{1}{m} \left( \frac{\Delta_3}{uy}\right)^{m} + \sum\limits_{m=1}^{t-1} \frac{1}{m} \left(\frac{\Delta_2}{wx}\right)^{m} -\sum\limits_{k=1}^{t-1} \frac{1}{k} \left(\frac{\Delta_1}{wy}\right)^{m} 
	\end{align*} 
	which, in $R/I^{t}$ is the same as 
	\begin{align*}
	&  \sum\limits_{m=1}^{\infty} \frac{1}{m} \left( \frac{\Delta_3}{uy}\right)^{m} + \sum\limits_{m=1}^{\infty} \frac{1}{m} \left(\frac{\Delta_2}{wx}\right)^{m} -\sum\limits_{k=1}^{\infty} \frac{1}{k} \left(\frac{\Delta_1}{wy}\right)^{m} \\
	=& \ln\left(\frac{vx}{uy}\right) + \ln\left( \frac{uz}{wx}\right) - \ln\left( \frac{vz}{wy}\right) \\ 
	=& \ln\left(\frac{vx}{uy}\right) + \ln\left( \frac{uz}{wx}\right) + \ln\left( \frac{wy}{vz}\right) \\ 
	=& \ln \left( \frac{vxuzwy}{uywxvz} \right)\\ 
	=& \ln(1) \\ 
	=& 0 . \qedhere
	\end{align*} 
		\end{enumerate} 
\end{proof} 
\subsection{The Element $\eta$ Is Not a Coboundary}\label{subsection:char0notImage}
 
	The element given by Table \ref{charzeroelement} is not a coboundary, that is, the element $\eta$ is not the image of an element from $\check{C}^2({\bf x}, R/I^t)$. 

\begin{proof} 
	\par Give $R$ the following multi-grading: 
	\begin{align*}
	\label{multigrading}  
	\deg(u)&=(1,0,0,0) & 	\deg(x)&=(1,0,0,1) \\
	\deg(v)&=(0,1,0,0) & 	\deg(y)&=(0,1,0,1) \\
	\deg(w)&=(0,0,1,0) & 	\deg(z)&=(0,0,1,1) 
	\end{align*}
	The generators of the ideal $I$ are homogeneous with respect to this multi-grading, and hence one obtains a grading on $R/I^t$.  
	The element, $\eta \in \check{C}^3({\bf x}, R/I^t)$ is homogeneous of degree $(0,0,0,0)$. 

	\par Suppose for the sake of contradiction that the element, $\eta$ were a coboundary, that is, in the image of $\check{C}^2({\bf x}, R/I^t)$. The map from $\check{C}^2({\bf x}, R/I^t)$ to $\check{C}^3({\bf x}, R/I^t)$ is degree-preserving. Therefore, an element of $\check{C}^2({\bf x}, R/I^t)$ mapping to the given element would necessarily be degree $(0,0,0,0)$ in each component. 
	\par Consider the $R_{uv}$ component of $\check{C}({\bf x}, R/I^t)$. 
	Since $\deg(\frac{1}{u^nv^m}) = (-n,-m,0,0)$, in order for an arbitrary element $\deg(\frac{a}{u^nv^m})$ of $(R/I^t)_{uv}$ with $a \in R/I^t$ to be of multi-degree $(0,0,0,0)$, it must be that $\deg(a) = (n, m, 0, 0)$, or $a=0$. But since $a \in R$, this means $a = \lambda u^n v^m$. Thus, any degree $(0,0,0,0)$ element in $R_{uv}$ is of the form $$ \frac{\lambda u^n v^m}{u^nv^m}$$
	for $\lambda \in \mathbb{F}$. 
	Thus any multi-degree $(0,0,0,0)$ element in $R_{uv}$ is, in fact, in the field $\mathbb{F}$. 
	\par The same argument shows that all elements in $R_{uw}, R_{vw}, R_{xy}, R_{xz}, R_{yz}, R_{ux}, R_{vy}$, and $R_{wz}$ in multi-degree $(0,0,0,0)$ are scalars from the field $\mathbb{F}$. 
	\par Note that $\eta$, is $0$ in the $R_{uvx}$ component. Since the elements in $R_{ux}$ and $R_{uv}$ are scalars in $\mathbb{F}$, the preimage of $\eta$ in the $R_{xv}$ component must also be a scalar.
	Using the fact that $\eta$ is 0 in the $R_{uvy}, R_{uwx}, R_{uwz}, R_{vwz}$, and $R_{vwy}$ components, a similar argument shows that the preimage of $\eta$ would be forced to be a scalar in every component of $\check{C}^2(\textbf{x},R)$.  
	
	\par Thus, if the element $\eta$ were in the image of an element of $\check{C}^2(\textbf{x}, R)$, it would be the image of an element that consisted of scalars in each component. However, $\frac{\Delta_1}{uy}$ is not in $\mathbb{F}$.
\end{proof} 
	\par Note that the arguments in \ref{subsection:char0notImage} were independent of the characteristic of the ground field. 
\section{Explorations in Characteristic $p > 0 $} \label{charPSection} 
\par The results of Theorem \ref{BBLSZ:CA} require the characteristic of the ground field be $0$. In characteristic $p>0$, the situation is remarkably different. 
\par We shall consider the same setting but over a field of positive characteristic. Let $X$ be a $2 \times 3$ matrix of indeterminates and let $R$ be the ring $\mathbb{F}[X]$ for a field of prime characteristic $p>0$:  
$$ R = \mathbb{F}\begin{bmatrix} u & v & w \\ x & y & z \end{bmatrix}. $$ 
As before, let $I$ be the ideal generated by size two minors of the matrix $X$, and let
$$  \Delta_1 = vz-wy, \quad \Delta_2 = wx-uz, \quad \Delta_3 = uy-vx. $$
In the characteristic 0 case, Proposition~\ref{rankOne} guaranteed that the local cohomology module $H^3_{\mf{m}}(R/I^t)_0$ is an $\mathbb{F}$-vector space of rank 1 for all $t \geq 2$. In the characteristic $p>0$ case, there is no such guarantee, and indeed,we shall see that the ranks of the local cohomology modules $H^3_{\mf{m}}(R/I^t)_0$ grow arbitrarily large on a subsequence of $t$ in the natural numbers. 

In Subsection \ref{subsection:charPconstruction}, we construct elements of $H^3_{\mf{m}}(R/I^t)_0$. The construction proceeds by showing in Subsection \ref{Subsection:notABoundaryCharP} that the given elements are not boundaries, then showing in Subsection \ref{subsection:cycleCharP} that the given elements are cycles. 
\subsection{Elements Of Local Cohomology Modules In Positive Characteristic} \label{subsection:charPconstruction} 
We seek to construct elements of $H^3_{\mf{m}}(R/I^t)_0$ when the ground field is characteristic $p > 0$. The element $\eta$ from Table \ref{charzeroelement} is no longer defined when the characteristic of the field is positive, as $\eta$ is defined using the fraction $\frac{1}{m}$ for arbitrary $m \leq t-1$.   
\begin{theorem} \label{charPRanks}
	Let $q$ be the largest power of p such that $q \leq t-1$, and let $q_2$ be the smallest power of $p$ such that $q+q_2 \geq t$. Further, suppose that $m$ is a positive integer with $0 < m \leq q$ such that $ q_2 \lvert m$. Then
	\begin{equation*}
	\rank H^3_{\mf{m}}(R/I^t)_0 \geq 2 \left \lfloor \frac{q}{q_2} \right \rfloor -1.
	\end{equation*} 
	In particular, whenever $t = p^e +1$ for some $e$, we have that the rank of $H^3_{\mf{m}}(R/I^t)_0$ is at least $2t-1$.
\end{theorem} 
\begin{proof} 

	The proof of Theorem \ref{charPRanks} proceeds by concrete construction; we will demonstrate  $2 \left \lfloor \frac{q}{q_2} \right \rfloor -1$ linearly independent elements in $H^3_{\mf{m}}(R/I^t)_0$.  
	\par First, let 
	\begin{equation*} 
	\alpha = \dfrac{vx}{uy}, \quad \beta = \dfrac{wy}{vz}, \quad \gamma = \dfrac{uz}{wx}.
	\end{equation*} 
	Note that $1-\alpha = \frac{\Delta_3}{uy}, 1- \beta = \frac{\Delta_1}{vz}$ and $1 - \gamma = \frac{\Delta_2}{wx}$. 
	We claim the elements of Table~\ref{charPElement} and Table~\ref{charPElement2} are nonzero elements of $H^3_{\mf{m}}(R/I^t)_0$ whenever $q$, $q_2$, and $m$ satisfy the hypotheses of Theorem~\ref{charPRanks}. Furthermore, the elements of Table~\ref{charPElement} and Table~\ref{charPElement2} are distinct when $m \neq q$. We will refer to the elements in Table \ref{charPElement} as $\eta_{1,m}$ and the elements in Table \ref{charPElement2} as $\eta_{2,m}$, with $m$ ranging over all integers less than $q$ that are divisible by $q_2$ as in Theorem \ref{charPRanks}. 
	
	\rowcolors{2}{babyBlossoms!50}{white}
	\begin{table}[t] \label{} 
		\begin{center}
			\caption{Elements in $\check{C}^3({\bf x}, \hat{R})$ in characteristic $p>0$} 
			\label{charPElement}
			\medskip
			\begin{tabular}{l r}
				\hline
				\rowcolor{babyBlossoms}
				Component of $\check{C}^{3}({\bf x},R/I^t)$ & \multicolumn{1}{c}{} \\
				\hline
				$R_{uvw}$ & $0$ \\ 
					$R_{uvx}$ & $0$ \\
				$R_{uwx}$ & $0$ \\ 
					$R_{vwx}$ & $0$ \\ 
				$R_{uvy}$ & $0$ \\ 
					$R_{uwy}$ & $0$ \\ 
				$R_{vwy}$ & $0$ \\ 
					$R_{uxy}$ & $\frac{x}{u}^{q-m}(\alpha^m-1)$ \\ 
				$R_{vxy}$ & $ \frac{y}{v}^{q-m}(\frac{1}{\alpha^m}-1)$\\
					$R_{wxy}$ & $ \frac{z}{w}^{q-m}(1 -\gamma^m - \frac{1}{\beta^m}+1)$\\
				$R_{uvz}$ & $0$ \\ 
					$R_{uwz}$ & $0$ \\ 
				$R_{vwz}$ & $0$ \\ 
					$R_{uxz}$ & $\frac{x}{u}^{q-m}(1-\frac{1}{\gamma}^m)$ \\ 
				$R_{vxz}$ & $\frac{y}{v}^{q-m}(\frac{1}{\alpha}^m-1+\beta^m-1))$  \\ 
					$R_{wxz}$ & $ \frac{z}{w}^{q-m}(1-\gamma^m)$ \\
				$R_{uyz}$ & $ \frac{x}{u}^{q-m}(1-\alpha^m - \frac{1}{\gamma}^m+1)$ \\ 
					$R_{vyz}$ & $ \frac{y}{v}^{q-m}(\beta^m - 1)$\\ 
				$R_{wyz}$ & $ \frac{z}{w}^{q-m}(\frac{1}{\beta}^m - 1)$ \\ 
					$R_{xyz}$ & $0$ \\ 
			\end{tabular} 
		\end{center}
	\end{table} 
		
	\rowcolors{2}{babyBlossoms!50}{white}
	\begin{table}[t]
		\begin{center}
			\caption{Elements in $\check{C}^3({\bf x}, \hat{R})$ in characteristic $p>0$} 
			\label{charPElement2}
			\medskip
			\begin{tabular}{l r}
				\hline
								\rowcolor{babyBlossoms}
				Component of $\check{C}^{3}({\bf x},R/I^t)$ & \multicolumn{1}{c}{} \\
				\hline
				$R_{uvw}$ & $0$ \\ 
					$R_{uvx}$ & $\frac{u}{x}^{q-m}(1-\frac{1}{\alpha}^m)$ \\
				$R_{uwx}$ & $\frac{u}{x}^{q-m}(1-\gamma^m)$ \\ 
					$R_{vwx}$ & $ \frac{u}{x}^{q-m}(1-\frac{1}{\alpha}^m - \gamma^m+1)$  \\ 
				$R_{uvy}$ &  $ \frac{v}{y}^{q-m}(\alpha^m-1)$ \\ 
					$R_{uwy}$ & $\frac{v}{y}^{q-m}(1-\alpha^m-\frac{1}{\beta}^m+1))$  \\ 
				$R_{vwy}$ &  $ \frac{v}{y}^{q-m}(\frac{1}{\beta}^m - 1)$ \\ 
					$R_{uxy}$ & $0$ \\ 
				$R_{vxy}$ & $0$\\
					$R_{wxy}$ & $0$ \\
				$R_{uvz}$ & $ \frac{w}{z}^{q-m}(1-\frac{1}{\gamma}^m - \beta^m+1)$ \\ 
					$R_{uwz}$ & $ \frac{w}{z}^{q-m}(\gamma + 1)$ \\ 
				$R_{vwz}$ & $ \frac{w}{z}^{q-m}(\beta^m - 1)$ \\ 
					$R_{uxz}$ &  $0$\\ 
				$R_{vxz}$ &  $0$ \\ 
					$R_{wxz}$ &  $0$\\
				$R_{uyz}$ & $0$\\ 
					$R_{vyz}$ &$0$\\ 
				$R_{wyz}$ &  $0$\\ 
					$R_{xyz}$ & $0$ \\ 
			\end{tabular} 
		\end{center}
	\end{table} 
	\par We consider the elements $\eta_{1, m}$. The argument for the elements $\eta_{2,m}$ is symmetric.  
	\subsection{The Elements $\eta_{1,m}$ Are Not Coboundaries} \label{Subsection:notABoundaryCharP}
 We first consider the case that $m \neq q$ and so $q - m \geq 1$. 
		We claim these elements are not in the image of $\check{C}^2({\bf x}, R/I^t)$. 
 \par Suppose the element $\eta_{1,m}$ is in the image of $\check{C}^2({\bf x}, R/I^t)$ for some arbitrary $m \neq q$ that satisfies the hypotheses of Theorem \ref{charPRanks}. Using the multi-grading introduced in Section \ref{subsection:char0notImage}, every component of the elements $\eta_{1,m}$  is of degree $(0,0,0, q-m)$, where $1\leq q-m<q$. Since the maps between \v{C}ech complexes are degree preserving, they would need to be the image of elements also of degree $(0,0,0,q-m)$. Consider $R_{xy}, R_{xz},$ and $R_{yz}$. There are no elements of degree $(0,0,0,q-m)$ in each of these components whenever $q-m \geq 1$. 
	\par If $\eta_{1,m}$ were in the image of an element in $\check{C}({\bf x}, R/I^t)$, the $R_{uvx}$ component would be a sum of the elements in the $R_{uv}, R_{ux}$, and $R_{vx}$ components. In the $R_{uv}, R_{ux}$, and $R_{vx}$ components, the multi-degree elements of degree $(0,0,0,q-m)$ where $q-m \geq 1$ are the following:
	\begin{align*} 
	R_{uv}: &\quad \left( \frac{y}{v}\right)^{q-m}\mathbb{F} \oplus \left( \frac{x}{u}\right)^{q-m}\mathbb{F} \\
	R_{ux}: &\quad \left( \frac{x}{u} \right)^{q-m} \mathbb{F} \\ 
	R_{vx}: &\quad \left( \frac{y}{v}\right)^{q-m} \mathbb{F}\left[\frac{uy}{vx}\right]. \\
	\end{align*} 
	
	The entry in the $R_{uvx}$ component of $\eta_{1,m}$ is $0$.
	As $\left( \frac{y}{v} \right)^{q-m} \mathbb{F} \oplus \left( \frac{x}{u} \right)^{q-m} \mathbb{F}$ do not contain polynomials in $\mathbb{F}[\frac{uy}{vx}]$, in order for the sum of components from $R_{uv}, R_{ux}$, and $R_{vx}$ to be zero, the element in the $R_{vx}$ component could not be a polynomial in $\mathbb{F}$ and must be of the form $\lambda \left(\frac{y}{v}\right)^{q-m}$ where $\lambda$ is some scalar from $\mathbb{F}$. 
	Thus, if $\eta_{1,m}$ is in the image of an element in $\check{C}^2({\bf x}, R/I^t)$, that element is a constant from $\mathbb{F}$ in the $R_{vx}$ component.
	\par Similarly, the element in $R_{wvz}$ comes from a sum of the elements in the $R_{wv}, R_{wz},$ and $R_{vz}$ components. Since the entry in the $R_{uvy}$ component is $0$, there must not be a nonzero power of $\frac{wy}{vz}$ in the $R_{vz}$ component. The entry in $R_{vz}$ then must be $\lambda' \left( \frac{y}{v} \right)$, where $\lambda'$ is some scalar from $\mathbb{F}$. 
	\par But then consider $\eta_{1,m}$ in the $R_{vxz}$ component
	$$ \left( \frac{y}{v}\right)^{q-m}\left(\frac{1}{\alpha^m}-1-\beta^m+1\right).$$
	If $\eta_{1,m}$ were a coboundary, then the above would have to be a sum of elements from $R_{vx}, R_{vz}$, and $R_{xz}$ components. We have already established there are no elements of degree $(0,0,0,q-~m)$ in $R_{xz}$, so it must only be from $R_{vx}$ and $R_{vz}$. However, by the above argument, the only possible elements in those components are in $\frac{y}{v}\mathbb{F}$. As $\eta_{1,m}$ in the $R_{vxz}$ component is not in $\frac{y}{v}\mathbb{F}$, it cannot be a coboundary.
	\subsection{The Elements $\eta_{1,m}$ Are Cocycles} \label{subsection:cycleCharP}
		 We show that the elements given by Table \ref{charPElement} are indeed cocycles, that is, the images of $\eta_{1, m}$ and $\eta_{2, m}$ are 0 in $\check{C}^4({\bf x}, R/I^t)$. The argument for the elements given by Table \ref{charPElement2} works similarly. 
\par	By symmetry, it suffices to check that the image of $\eta_{1,m}$ is 0 in the three components, $(R/I^t)_{uvxy}, (R/I^t)_{uxyz},$ and $(R/I^t)_{uwxy}$.
	\par First consider the image of $\eta_{1,m}$ in the $(R/I^t)_{uxyz}$ component: 
	\begin{align*} 
	-&\left(\frac{x}{u}\right)^{q-m}(\alpha^m-1)
	+\left(\frac{x}{u}\right)^{q-m}\left(1-\frac{1}{\gamma^m}\right)
	-\left(\frac{x}{u}\right)^{q-m}\left(1-\alpha^m - \frac{1}{\gamma^m}+1\right) \\ 
	=&\left( \frac{x}{u} \right)^{q-m}\left( -\alpha^m + 1 + 1 - \frac{1}{\gamma^m} - 1 + \alpha^m + \frac{1}{\gamma^m} - 1 \right) \\ 
	=&0.
	\end{align*} 
	\par Second, consider the image of $\eta_{1,m}$ in the $(R/I^t)_{uwxy}$ component:
	\begin{align*}
	& - \left( \frac{x}{u}\right)^{q-m}(\alpha^m-1)  + \left( \frac{z}{w}\right)^{q-m}(1-\gamma^m- \frac{1}{\beta^m}+1)\\ 
	=& -\left( \frac{x}{u}\right)^{q-m} \alpha^m+\left( \frac{x}{u}\right)^{q-m} +  \left( \frac{z}{w}\right)^{q-m}  - \left( \frac{z}{w}\right)^{q-m} \gamma^m - \left( \frac{z}{w} \right)^{q-m}\frac{1}{\beta^m} + \left( \frac{z}{w} \right)^{q-m} \\ 
	=& -\left( \frac{x}{u}\right)^{q-m} \left( \frac{vx}{uy} \right)^m+\left( \frac{x}{u}\right)^{q-m} +  \left( \frac{z}{w}\right)^{q-m} 
	\\ & \quad \quad \quad  - \left( \frac{z}{w}\right)^{q-m} \left( \frac{uz}{wx} \right)^m - \left( \frac{z}{w} \right)^{q-m}\left( \frac{vz}{wy}\right)^m + \left( \frac{z}{w} \right)^{q-m} \\ 
	=& -\frac{x^q v^m}{u^q y^m} + \left(\frac{x}{u}\right)^{q-m}+ \left(\frac{z}{w} \right)^{q-m} - \frac{z^q u^m}{w^qx^m} - \frac{v^m z^q}{w^q y^m} + \left( \frac{z}{w}\right)^{q-m}.
	\end{align*} 
	Since $m$ is divisible by $q_2$ by hypothesis, write $m = \zeta q_2$ for some natural number $\zeta$. Then the above equals
	\begin{align*} 
	& \frac{(wx)^q(-u^{\zeta q_2}y^{\zeta q_2 } + v^{\zeta q_2 }x^{\zeta q_2}) + (uz)^q(u^{\zeta q_2}y^{\zeta q_2} - v^{\zeta q_2} x^{\zeta q_2} )}{u^q w^q x^m y^m}\\ 
	=& \frac{(wx)^q(-u^{\zeta}y^{\zeta} + v^{\zeta }x^{\zeta})^{q_2} + (uz)^q(u^{\zeta}y^{\zeta} - v^{\zeta} x^{\zeta} )^{q_2}}{u^q w^q x^m y^m} \\
	=& \frac{-(wx)^q(v^{\zeta }x^{\zeta}- u^{\zeta}y^{\zeta})^{q_2} + (uz)^q(u^{\zeta}y^{\zeta} - v^{\zeta} x^{\zeta} )^{q_2}}{u^q w^q x^m y^m} \\
	=& \frac{\left(-(wx)^q + (uz)^q\right) (v^{\zeta }x^{\zeta}- u^{\zeta}y^{\zeta})^{q_2}}{u^q w^q x^m y^m} \\
	=&  \frac{\left(-wz + uz \right)^q (v^{\zeta }x^{\zeta}- u^{\zeta}y^{\zeta})^{q_2}}{u^q w^q x^m y^m}. 
	\end{align*} 
	The polynomial $(x-y)$ divides $(x^{\zeta} - y^{\zeta})$ for any natural number $\zeta$. 
	Let $\phi_{\zeta}(x,y)$ be the polynomial such that $\phi_{\zeta}(x,y)(x-y) = (x^{\zeta} - y^{\zeta})$. Then the above is 
	\begin{align*} 
	&\frac{\left(-wz + uz \right)^q ((uy-vx) \phi_{\zeta}(uy,vx))^{q_2}}{u^q w^q x^m y^m} \\
	&= \frac{\left( \Delta_1 \right)^q \left( \Delta_3 \right)^{q_2} \phi_{\zeta}(uy,vx)^{q_2} }{u^q w^q x^m y^m}. 
	\end{align*} 
	By hypothesis, $q + q_2 \geq t$. So this is indeed 0 in the $R_{uwxy}$ component of $\check{C}^4({\bf x}, R/I^t)$. 
	
	Third and finally, consider the image in $R_{uvxy}$. We will use the fact that $m$ is divisible by $q_2$, so again, let $\zeta$ be the natural number such that $m= \zeta q_2$.
	\begin{align*} 
	&\left( \frac{x}{u}\right)^{q-m}\left(1-\left(\frac{vx}{uy}\right)^m\right) -  \left(\frac{y}{v}\right)^{q-m}\left(\left( \frac{uy}{vx} \right)^m-1\right)   
	\\ =&\left( \frac{x}{u}\right)^{q-m}\left(1-\left(\frac{vx}{uy}\right)^{\zeta q_2}\right) -  \left(\frac{y}{v}\right)^{q-m}\left(\left( \frac{uy}{vx} \right)^{\zeta q_2}-1\right)  
	\\ =&\left( \frac{x}{u}\right)^{q-m}\left(1-\left(\frac{vx}{uy}\right)^{\zeta }\right)^{q_2} -  \left(\frac{y}{v}\right)^{q-m}\left(\left( \frac{uy}{vx} \right)^{\zeta}-1\right)^{q_2}  
	\\ =&\left( \frac{x}{u}\right)^{q-m}\left(\frac{1}{uy}\right)^m\left((uy)^\zeta-\left(vx\right)^{\zeta }\right)^{q_2} - \left(\frac{y}{v}\right)^{q-m}\left(\frac{1}{vx} \right)^{m}\left(\left( uy \right)^{\zeta}-\left(vx\right)^{\zeta}\right)^{q_2}. 
	\end{align*} 
 Then the above is \\
	\begin{align*}
	&\left( \frac{x}{u}\right)^{q-m}\left(\frac{1}{uy}\right)^m (uy - vx)^{q_2} \phi_{\zeta}^{q_2}(uy, vx) - \left(\frac{y}{v}\right)^{q-m}\left(\frac{1}{vx} \right)^{m} (uy-vx)^{q_2} \phi_{\zeta}^{q_2}(uy,vx) 
	\\ &= \Delta_3^{q_2} \phi^{q_2}_{\zeta}(uy,vx) \left( \left(\frac{x}{u}\right)^{q-m} \left(\frac{1}{uy}\right)^m - \left(\frac{y}{v}\right)^{q-m}\left(\frac{1}{vx}\right)^m \right)  
	\\ &= \Delta_3^{q_2} \phi^{q_2}_{\zeta}(uy,vx) \left(\frac{1}{uv}\right)^{q-m}\left( \left(xv\right)^{q-m} \left(\frac{1}{uy}\right)^m - \left(uy\right)^{q-m}\left(\frac{1}{vx}\right)^m \right)  
	\\ &= \Delta_3^{q_2} \phi^{q_2}_{\zeta}(uy,vx) \left(\frac{1}{uv}\right)^{q-m} \left(\frac{1}{uyvx}^m\right)\left( \left(vx\right)^{q} - \left(uy\right)^{q}\right)  
	\\ &= \Delta_3^{q_2} \phi^{q_2}_{\zeta}(uy,vx) \left(\frac{1}{uv}\right)^{q-m} \left(\frac{1}{uyvx}^m\right)\left( vx - uy\right)^q  
	\\ &= \Delta_3^{q_2 + q}  \phi^{q_2}_{\zeta}(uy,vx) \left(\frac{1}{uv}\right)^{q-m} \left(\frac{1}{uyvx}^m\right).
	\end{align*}
	Recall that $q_2 + q \geq t$, so this is indeed $0$ in the $R_{uvxy}$ component of $\check{C}^4({\bf x},R/I^t)$. 
	
	\par Now consider the case that $m=q$ and so $q-m = 0$, which is recorded in Table~\ref{charPchar0unite}. In this case, the element is in multi-degree $(0,0,0,0)$.  Recall that the argument from Section~\ref{subsection:char0notImage} is characteristic free, so to show this element is not a coboundary.  
		
	\rowcolors{2}{babyBlossoms!50}{white}
	\begin{table}[t]
		\begin{center}
			\caption{Elements in $\check{C}^3({\bf x}, \hat{R})$ in characteristic $p>0$} 
			\label{charPchar0unite}
			\medskip
			\begin{tabular}{l r | r}
								\rowcolor{babyBlossoms}
				\hline
				Component of $\check{C}^{3}({\bf x},R/I^t)$ & \multicolumn{1}{c}{} &\\
				\hline
				$R_{uvw}$ & $0$ & $0$ \\ 
					$R_{uvx}$ & $0$ & $0$\\
				$R_{uwx}$ & $0$ & $0$\\ 
					$R_{vwx}$ & $0$ & $0$\\ 
				$R_{uvy}$ & $0$ & $0$\\ 
					$R_{uwy}$ & $0$ & $0$ \\ 
				$R_{vwy}$ & $0$ & $0$ \\ 
					$R_{uxy}$ & $(\alpha^q-1)$ & $- \left( \frac{\Delta_3}{uy} \right)^q$ \\ 
				$R_{vxy}$ & $ (\frac{1}{\alpha^q}-1)$ & $\left( \frac{\Delta_3}{vx} \right)^q$\\
					$R_{wxy}$ & $ (1 -\gamma^q - \frac{1}{\beta^q}+1)$ & $\left(\frac{\Delta_2}{wx}\right)^q - \left(\frac{\Delta_1}{wy}\right)^q$ \\
				$R_{uvz}$ & $0$ & $0$\\ 
					$R_{uwz}$ & $0$ & $0$\\ 
				$R_{vwz}$ & $0$ & $0$ \\ 
					$R_{uxz}$ & $(1-\frac{1}{\gamma}^q)$ & $- \left( \frac{\Delta_2}{uz}\right)^q $ \\ 
				$R_{vxz}$ & $(\frac{1}{\alpha}^q-1+\beta^q-1))$ &  $\left( \frac{\Delta_3}{vx} \right)^q -\left( \frac{\Delta_1}{vz}  \right)^q$\\ 
					$R_{wxz}$ & $(1-\gamma^q)$ & $\left( \frac{\Delta_2}{wx}\right)^q $\\
				$R_{uyz}$ & $(1-\alpha^q - \frac{1}{\gamma}^q+1)$ & $ \left( \frac{\Delta_3}{uy}\right)^q -\left( \frac{\Delta_2}{uz} \right)^q $\\ 
					$R_{vyz}$ & $(\beta^q- 1)$ & $-\left( \frac{\Delta_1}{vz}\right)^q $\\ 
				$R_{wyz}$ & $(\frac{1}{\beta}^q - 1)$ & $\left( \frac{\Delta}{wy}\right)^q $\\ 
					$R_{xyz}$ & $0$ & $0$\\ 
			\end{tabular} 
		\end{center}
	\end{table} 
	Thus we have demonstrated $2\lfloor \frac{q}{q_2} - 1 \rfloor $ linearly independent elements of $H^3_{\mf{m}}(R/I^t)_0$ when the base field is characteristic $p>0$.
\end{proof} 

\pagebreak 

\bibliographystyle{alpha} 

\bibliography{main}

\end{document}